%% file: Constrained_Interpolation_arxiv.tex
\documentclass{amsart}

\usepackage{amsrefs}
\usepackage{mathrsfs}
\usepackage{amssymb}
\usepackage{cite}
\usepackage[all]{xy}
\usepackage{graphicx}
\usepackage{textcomp}
\usepackage{charter}
\usepackage[vflt]{floatflt}
\usepackage{mathtools}
\usepackage{enumerate}
\usepackage{wasysym}
\usepackage{rotating}
\usepackage{pdflscape}
\usepackage{fullpage}
\usepackage{bm}
\usepackage{yhmath}
\usepackage[pdftex,colorlinks]{hyperref}
\hypersetup{
pdfauthor={Shamovich, E.},
bookmarksnumbered,
pdfstartview={FitH}}%

\newtheorem{thm}{Theorem}[section]
\newtheorem{prop}[thm]{Proposition}
\newtheorem{lem}[thm]{Lemma}
\newtheorem{cor}[thm]{Corollary}
\newtheorem*{thm*}{Theorem}
\newtheorem*{cor*}{Corollary}
\newtheorem{quest}[thm]{Question}

\theoremstyle{definition}
\newtheorem{dfn}[thm]{Definition}

\newtheorem{example}[thm]{Example}
\newtheorem{rem}[thm]{Remark}

\theoremstyle{remark}

\numberwithin{equation}{section}

\include{operators}
\include{fonts}

\newcommand{\ugrp}{\left(\cP_d/\fc\right)^{\times}}
\newcommand{\augrp}{\left(A/\fc\right)^{\times}}
\newcommand{\Pic}{\operatorname{Pic}}

\def\mcc{M\raise.5ex\hbox{c}C}
\def\mccarthy{M\raise.5ex\hbox{c}Carthy}

\begin{document}

\title{Nevanlinna-Pick Families and Singular Rational Varieties}

\author{Kenneth R. Davidson}
\address{Pure Math.\ Dept., University of Waterloo, Waterloo, ON, N2L 3G1, Canada}
\email{krdavidson@uwaterloo.ca}

\author{Eli Shamovich}
\email{eshamovich@uwaterloo.ca}

\begin{abstract}
The goal of this note is to apply ideas from commutative algebra (a.k.a. affine algebraic geometry) to the question of constrained Nevanlinna-Pick interpolation. More precisely, we consider subalgebras $A \subset \C[z_1,\ldots,z_d]$, such that the map from the affine space to the spectrum of $A$ is an isomorphism except for finitely many points. Letting $\fA$ be the weak-$*$ closure of $A$ in $\cM_d$ -- the multiplier algebra of the Drury-Arveson space. We provide a parametrization for the Nevanlinna-Pick family of $M_k(\fA)$ for $k \geq 1$. In particular, when $k=1$ the parameter space for the Nevanlinna-Pick family is the Picard group of $A$. 
\end{abstract}

\maketitle

\newcommand{\symball}{\widetilde{\fB_d}}
\newcommand{\PU}{\operatorname{PU}}
\newcommand{\sat}{\operatorname{sat}}

\section{Introduction}

Let $\D$ denote the unit disc in the complex plane. The Nevanlinna-Pick interpolation problem is to find an analytic function $f \colon \D \to \D$, such that at the prescribed set of points $F = \left\{z_1,\ldots,z_n\right\} \subset \D$, the function $f$ attains some prescribed values, namely $f(z_j) = w_j$, with $w_1,\ldots,w_n$ given. A clean and elegant solution for this problem was obtained by Pick \cite{Pick16} and Nevanlinna \cite{Nev19,Nev29}. Such a function $f$ exists if and only if the Pick matrix $\left[ \frac{1 - w_i \overline{w_j}}{1 - z_i \overline{z_j}} \right]_{i,j =1}^n$ is positive semi-definite. In fact, the same still holds if one replaces the scalars $w_i$ with matrices.

In \cite{Sar67} Sarason gave an operator theoretic interpretation to the Nevanlinna-Pick interpolation problem. He observed that the problem can be formulated in terms of a commutant lifting problem on a subspace of the Hardy space $H^2 = H^2(\D)$. Note that the expression $k(z,w) = \frac{1}{1 - z \bar{w}}$ that appears in the Pick matrix is in fact the Szego kernel, the reproducing kernel of $H^2$. The multiplier algebra of $H^2$, namely the algebra of all functions $f$ such that  $fH^2 \subset H^2$, turns out to be $H^{\infty}$, the algebra of all bounded functions on the disc. Let $B_F$ be the Blaschke product that vanishes on $F$. It is a bounded function and thus a multiplier on $H^2$. Furthermore, multiplication by $B_F$ is an isometry on $H^2$.  Thus we can consider the subspace $M_F = H^2 \ominus B_F H^2$. Let $P_{M_F}$ be the orthogonal projection on $M_F$. For every $f \in H^{\infty}$, let $M_f$ be the multiplication operator defined by $f$ on $H^2$. Let $\cI_F \subset H^{\infty}$ be the ideal of functions that vanish on $F$. Sarason showed that the map $f \mapsto P_{M_F} M_f P_{M_F}$ induces an isometry on $H^{\infty}/\cI_F$.

Abrahamse in \cite{Abr79} replaced the disc with a multiply connected domain in $\Omega \subset \C$. He showed that in this case one cannot consider a single reproducing kernel Hilbert space, but one must consider a family of kernels parametrized by the torus $\R^g/\Z^g$, where the connectivity of $\Omega$ is $g+1$. Each point of the parameter space corresponds to a character $\chi$ of the fundamental group $\pi_1(\Omega)$ and we associate to $\chi$ the function space of $\chi$-automorphic functions on $\Omega$. These spaces were considered in the case of the annulus by Sarason in \cite{Sar65}. The work of Abrahamse generated a lot of interest (see for example \cite{BalCla96, VinFed98} for the treatment of the question of which subspaces of the parameter space are sufficient for a particular interpolation datum and what subspaces are sufficient for all data). The term Nevanlinna-Pick family was coined to describe the situation where a family of kernels is sufficient to solve the Nevanlinna-Pick problem.

The first author, Paulsen, Raghupathi and Singh in \cite{DPRS09} observed that Nevanlinna-Pick families arise naturally in the setting of the disc, provided one considers a constrained interpolation question. They  consider $H^{\infty}_1$,  the algebra of all bounded analytic functions on the disc with vanishing derivative at the origin. They prove that there is a Nevanlinna-Pick family  parametrized by the projective $2$-sphere. Following their work, Raghupathi generalized their results to the case of algebras of the form $\C 1 + B H^{\infty}$, where $B$ is a Blaschke product \cite{Rag09}, and proved that constrained interpolation on the disc combined with an action of a Fuchsian group yields Abrahamse's result in \cite{Rag09-2}. In \cite{BBtH10} Ball, Bolotnikov and ter Horst have generalized the result of \cite{DPRS09} to the matrix-valued case and showed that one can parametrize the Nevanlinna-Pick families for $M_k(H^{\infty}_1)$ by a disjoint union of Grassmannians. A dual parametrization using test functions was provided by Dritschel and Pickering \cite{DriPic12} in the case of the Neil parabola and much more generally by Dritschel and Undrakh \cite{DriUnd18}. The former was used by Dritschel, Jury and McCullough \cite{DJM16} to study rational dilations on the Neil parabola.

Finally, the first author jointly with Hamilton in \cite{DavHam11} used the predual factorization property $\A_1(1)$ to show that every weak-$*$ closed subalgebra $\fA$ of the multiplier algebra of a complete Nevanlinna-Pick space has a Nevanlinna-Pick family. In particular, they have considered the algebra $\cM_d$ of multipliers on the Drury-Arveson space $H^2_d$. The Drury-Arveson space is a complete Nevanlinna-Pick space by \cite{DavPit98b, AglMcC00}. Arias and Popescu \cite{AriPop00} show that $\cM_d$ has the property $\A_1(1)$.  However, the parametrization in \cite{DavHam11} given for the Nevanlinna-Pick family of a weak-$*$ closed subalgebra $\fA \subset \cM_d$ is by all $\cM_d$-cyclic vectors in $H^2_d$.

The goal of this note is to provide a concrete description for the parameters of Nevanlinna-Pick families for a certain class of weak-$*$ closed subalgebras of $\cM_d$. Note that the algebra $H^{\infty}_1$ is generated by $z^2$ and $z^3$ and can be viewed as the pullback of bounded analytic functions on the cuspidal cubic (also known as the Neil parabola) in $\D^2$. Cusp algebras such as this one were considered first by Agler and McCarthy in \cite{AglMcC09}. Let us consider algebras $A \subset \C[z_1,\ldots,z_d] = \cP_d$, such that the induced map from the affine space to the spectrum of $A$ is a biholomorphism outside finitely many points. The key idea is to consider the largest common ideal in $\fc \subset A \subset \cP_d$, called the conductor ideal (see Section \ref{sec:conductor} for the definition). By analogy to \cite{Abr79}, we want to consider line (and more generally vector) bundles on the spectrum of $A$. Fortunately, we can describe them rather easily by considering the free $A/\fc$-submodules $M \subset \left(\cP_d/\fc\right)^{\oplus k}$, such that $M \otimes \left( \cP_d/\fc\right) \cong \left(\cP_d/\fc\right)^{\oplus k}$ (compare to \cite{BDG01}). We set $\fA \subset \cM_d$ to be the weak-$*$ closure of $A$. 

In Section \ref{sec:conductor} we will collect some necessary information about the conductor ideal in the topological setting. In Section \ref{sec:scalar_case} we show that the parameter space for the Nevanlinna-Pick family is the Picard group of $A$, i.e, the group of isomorphism classes of line bundles on the spectrum of $A$. In Section \ref{sec:matrix_case} we will prove that $M_k(\cM_d)$ has the property $\A_1(1)$ acting on $M_k(H^2_d)$ by transposed right multiplication and provide a description of the parameter space for the Nevanlinna-Pick families of $M_k(\fA)$ for $k \geq 2$. Lastly, in Section \ref{sec:a11_mat} we prove the $M_6(\cM_d)$ does not have $\A_1(1)$ acting on $\cH^2_d \otimes \C^6$. This implies that one cannot avoid considering vector bundles on the singular variety if one wants to do constrained matrix-valued Nevanlinna-Pick interpolation.

\section{The Conductor Ideal} \label{sec:conductor}

The idea of conductor ideal is quite old in commutative algebra and algebraic number theory. Let $A \subset B$ be two rings. Then the conductor of $A$ in $B$ is the ideal of elements $a \in A$, such that $a B \subset A$. Alternatively, the conductor is the annihilator of the $A$-module $B/A$. It turns out that the conductor is an ideal of $B$ and it is the largest ideal common to $A$ and $B$. 

We are going to discuss topological algebras, hence we make these defintions in our setting. All algebras considered in this section are unital and commutative. Let $\fA \subset \fB \subset B(\cH)$ be two weak-$*$ closed operator algebras. We assume that $\fA$ is a closed subalgebra of $\fB$. 

\begin{dfn} \label{dfn:conductor}
The conductor ideal is the ideal $\fc \subset \fA$  defined by  
\[ \fc = \{ f \in \fA \mid f \fB \subset \fA \} . \]
\end{dfn} 

For every $g \in \fB$ and every $f \in \fc$ we have  $f g \fB \subset f \fB \subset \fA$. Thus $\fc$ is an ideal of $\fB$. Furthermore, it is clearly the largest ideal with this property.

\begin{lem} \label{lem:cond_sq}
Let $\fA \subset \fB$ be as above and let $\fc$ be the conductor. Then $\fc$ is weak-$*$ closed and we have the commutative diagram 
\begin{equation} \label{eq:conductor_square}
\xymatrix{\fA \ar[r] \ar[d]_{\pi} & \fB \ar[d]^{\pi} \\ \fA/\fc \ar[r] & \fB/\fc.}
\end{equation}
The vertical arrows are quotient maps and the horizontal ones are embeddings. Furthermore, $\fA = \pi^{-1}(\fA/\fc)$ as a subspace of $\fB$.
\end{lem}

\begin{proof}
Let $a_{\alpha} \to a$ be a weak-$*$ convergent net, with $a_{\alpha} \in \fc$. Since multiplication is separately weak-$*$ continuous,  for every $b \in \fB$, $a_{\alpha} b \to a b$. Since $A$ is closed $a b \in \fA$. Hence $a \in \fc$.

Now it is clear that $\fA \subset \pi^{-1}(\fA/\fc)$, so let $f \in \pi^{-1}(A/\fc)$. Hence there exists $a \in \fA$, such that $f - a \in \fc$, but $\fc \subset A$ and we are done.
\end{proof}

\newcommand{\lat}{\operatorname{Lat}}

Recall that $\lat(\fA)$ denotes the lattice of $\fA$-invariant subspaces of $\cH$. Let us write $\cC = \overline{\fc \cH}$. It is clear that $\cC$ is both $\fA$ and $\fB$-invariant subspace. Hence, the compression map $\fB \to B(\cC^{\perp})$ is a completely contractive and weak-$*$ continuous homomorphism. The kernel of this map contains $\fc$ and thus makes $\cC^{\perp}$ into a $\fB/\fc$-module (and also an $\fA/\fc$-module).

\begin{dfn} \label{dfn:H_M}
Let $\cN \subset \cC^{\perp}$ be an $\fA/\fc$-invariant subspace. We define $\cH_{\cN} = \cN \oplus \cC$. It is clear that $\cH_{\cN} \in \lat(\fA)$.
\end{dfn}

Let us assume that the kernel of the compression to $\cC^{\perp}$ is precisely $\fc$. This is valid for $\cM_d$, the multiplier algebra of Drury-Arveson space as the following lemma shows.

\begin{lem} \label{lem:DA compression}
If $\cJ$ is an ideal of $\cM_d$, the multiplier algebra of Drury-Arveson space, then compression to $\overline{\cI H^2_d}^\perp$ is completely isometrically isomorphic to $\cM_d/\cJ$.
\end{lem}

\begin{proof}
We use the fact that $\cM_d$ is equal to the compression of the non-commutative analytic Toeplitz algebra $\cL_d$ (acting on full Fock space $\cF^2_d$) to symmetric Fock space, which is the orthogonal complement of the range of the commutator ideal $\cC$.
This identification is a special case of \cite[Theorem 2.1]{DavPit98a}.
This result shows that if $\cI \subset \cL_d$ is a weak-$*$ closed ideal, then the kernel of the compression onto the complement of its range is precisely $\cI$ and the quotient is completely isometric isomorphic to $\cL_d/\cI$. 
The corresponding result for $\cM_d$ holds for weak-$*$ closed ideals $\cJ$ of $\cM_d$ by applying the result to the preimage of $\cJ$ in $\cL_d$.
\end{proof}

Note that if $\xi \in \cH$ is $\fB$-cyclic, then $\cC = \overline{\fc \xi}$ and the cyclic module $\overline{\fA \xi} = \overline{\fA/\fc} P_{\cC^{\perp}} \xi \oplus \cC$, where $P_{\cC^{\perp}}$ is the orthogonal projection onto $\cC^{\perp}$. We record this in the following proposition.

\begin{prop} \label{prop:cyclic}
Assume that $\fc$ is precisely the kernel of the compression of $\fB$ to $\cC^{\perp}$. Then the cyclic modules $\fA \xi$, where $\xi$ is $\fB$-cyclic are in one-to-one correspondence with cyclic $\fA/\fc$-submodules of $\cC^{\perp}$.
\end{prop}

\section{Nevanlinna-Pick Families --- The Scalar Case} \label{sec:scalar_case}

Let $H^2_d$, for $d \in \N$, denote the Drury-Arveson space. Recall that this space is a reproducing kernel space of analytic functions on $\B_d$, the unit ball of $\C^d$, with reproducing kernel $k_d(z,w) = \frac{1}{1 - \langle z,w \rangle}$. We will denote by $\cM_d$ the algebra of multipliers of $H^2_d$. In particular, for $d=1$, $H^2_1 = H^2(\D)$ and $\cM_1 = H^{\infty}(\D)$. We will write simply $H^2$ and $H^{\infty}$ for $H^2_1$ and $\cM_1$, respectively. Note that since $\cM_d$ is a multiplier algebra of a reproducing kernel HIlbert space, it is weak-$*$ closed. Let $\fA \subset \cM_d$ be a weak-$*$ closed subalgebra. We fix a set of points $F = \left\{z_1,\ldots,z_n \right\} \subset \B_d$ and assume for simplicity that $\fA$ separates $F$. Let $\cI_F \subset \fA$ be the weak-$*$ closed ideal of functions that vanish on $F$. We say that $\xi \in H^2_d$ is outer if it is $\cM_d$-cyclic. For the convenience of the reader we recall some material from \cite{DavHam11}.

\begin{dfn}
We say that a weak-$*$ closed subalgebra $\fA \subset B(\cH)$ has property $\A_1(1)$, or alternatively that $\fA$ has property $\A_1(1)$ acting on $\cH$, if for every weak-$*$ continuous functional $\varphi$ on $\fA$ with $\|\varphi\| < 1$, there exist $\xi ,\eta \in \cH$, such that $\|\xi\|, \|\eta\| < 1$ and $\varphi(f) = \langle f \xi, \eta \rangle$ for every $f \in \fA$.
\end{dfn}

The following result was established in \cite[Proposition 6.2]{AriPop00} and an alternative argument is found in \cite[Theorem 5.2]{DavHam11}.

\begin{thm}[Arias-Popescu] \label{thm:dav_ham}
Let $\fA \subset \cM_d$ be a weak-$*$ closed operator algebra. Then $\fA$ has $\A_1(1)$ and furthermore, $\xi$ may be chosen to be outer.
\end{thm}

As was shown in \cite{Abr79}, \cite{DPRS09} and \cite{Rag09}, to understand constrained interpolation we need to consider families of kernels. By \cite[Lemma 2.1]{DavHam11}, every cyclic $\fA$-submodule $L \subset H^2_d$ is a reproducing kernel Hilbert space with respect to a kernel $k^L$ defined on all of $\B_d$.

\begin{dfn}[Davidson-Hamilton]
Let $\fA \subset \cM_d$ be a weak-$*$ closed subalgebra. We say that a collection of kernels $\{k^{L_j}\}_{j \in J}$ associated with cyclic $\fA$-submodules is a Nevanlinna-Pick family, if for every set of points $F = \{z_1,\ldots,z_n\} \subset \B_d$ separated by $\fA$ and complex scalars $w_1,\ldots,w_n$, there exists $f \in \fA$, such that $f(z_{\ell}) = w_{\ell}$, for all $\ell =1,\ldots,d$ and $\|f\| \leq 1$ if and only if the Pick matrices
\[
\left[ (1 - w_k \overline{w_{ell}}) k^{L_j}(z_k,z_{\ell}) \right]_{k,\ell = 1}^n 
\]
are positive for every $j \in J$.
\end{dfn}

By \cite[Theorem 5.5]{DavHam11} every weak-$*$ closed subalgebra of $\cM_d$ admits a Nevanlinna-Pick family parametrized by the outer functions. This parameter family is not described in detail and our goal is to show that for finitely many constraints, this family is parametrized by a finite-dimensional manifold.

Set $\cP_d = \C[z_1,\ldots,z_d]$. Let $\psi \colon \C^d \to \C^e$ be a polynomial map that is an isomorphism outside a finite set of points. Let $A \subset \cP_d$ be the corresponding algebra, i.e, the image of $\cP_e$ under $\psi^*$ and $\fA = \overline{A}^{w*} \subset \cM_d$. The spectrum of $A$ is the ring of polynomial functions on the image of $\C^d$. This is a rational variety with finitely many singular points. Let $\alpha_1,\ldots,\alpha_{\ell}$ be the points in the fibers over the singular points in the image. Let $\fm_j$ be the maximal ideal corresponding to $\alpha_j$. Recall that if $\fp \subset \cP_d$ is a prime ideal, an ideal $\fq$ is called $\fp$-primary, if the only associated prime of $\cP_d/\fq$ is $\fp$. By \cite[Proposition 3.9]{Eis95} we have that there exists $k \in \N$, such that $\fp^k \subset \fq \subset \fp$. By primary decomposition there exist $\fm_j$-primary ideals $\fq_j$, such that $\fc = \fq_1 \cap \cdots \cap \fq_{\ell}$. Thus by the Chinese remainder theorem $\cP_d/\fc \cong \cP_d/\fq_1 \times \cdots \cP_d/\fq_{\ell}$. Now $\cP_d/\fq_j$ is naturally embedded into $\cP_d/\fm_j^{k_j}$ and thus the dual space of $\cP_d/\fq_j$ can be identified with the quotient of the space spanned by the functionals taking the value of the polynomial and its derivatives up-to total order $k_j$ at the point $\alpha_j$.

From now on we will assume that all the points of the support of $\fc$ lie in $\fB_d$, the unit ball of $\C^d$. Let $\fc \subset A$ be the conductor ideal of $A$.  Then the conductor of $\fA$ in $\cM_d$ is the weak-$*$ closure of $\fc$.

\begin{lem} \label{lem:no_more_polys_in_closure}
\strut\qquad\qquad$
A = \cP_d \cap \fA.
$
\end{lem}
\begin{proof}
One inclusion is obvious. On the other hand the dual space of $\cP_d/\fc$ is spanned by the functionals of evaluation of the polynomial and its derivatives at the points $\alpha_j$. The fiber square \eqref{eq:conductor_square} tells us that $A$ is precisely the preimage of $A/\fc$ in $\cP_d/\fc$. Since $A/\fc$ is a subspace of $\cP_d/\fc$ it is given by the vanishing of finitely many functionals. If $f \in \cP_d \setminus A$, then there exists a functional in $\left(\cP_d/\fc\right)^*$ that vanishes on $A/\fc$ and does not vanish on the image of $f$. Now note that the evaluation functionals and the valuation of derivatives at points inside the disk are weak-$*$ continuous functionals on $\cM_d$. Therefore we can conclude that if $f \in \cP_d \setminus A$, then  $f \notin \fA$.
\end{proof}

The following is a generalization of the Helson-Lowdenslager Theorem in the case when $d=1$ (compare with \cite[Theorem 1.3]{DPRS09} in the case of $H^{\infty}_1$).

\begin{thm} \label{thm:our_hl}
Let $d=1$ and write $\cP_1 = \C[z]$ and $\cM_1 = H^{\infty}$. Let $\cN \subset L^2(\T)$ be an $\fA$-invariant closed subspace, that is not invariant for $H^{\infty}$. Then there exists an $A/\fc$-submodule $M \subset \C[z]/\fc$, such that $\cN = J H^2_M$, where $J$ is an unimodular function.
\end{thm}

\begin{proof}
The proof is essentially the proof in \cite{DPRS09}. Set $\widetilde{\cN} = \overline{H^{\infty} \cN}$. Recall that every ideal in $\C[z]$ is principal and thus $\fc$ is generated by $f_c = \prod_{j=1}^k (z - \alpha_j)^{m_j}$. Additionally, we have a Blaschke product $B_c$, such that $\cC = B_c H^2$. By the classical Helson-Lowdenslager theorem, either $\widetilde{\cN}$ is $L^2(E)$ for some measurable $E \subset \T$ or $\widetilde{\cN} = J H^2$ for some unimodular function $J$. In the former case, note that $z$ is a unitary on $\widetilde{\cN}$ and thus $f_c$ acts as an invertible operator. We conclude that $\widetilde{\cN} = f_c \widetilde{\cN} \subset \cN$ and this contradicts our assumption that $\cN$ is not $H^{\infty}$-invariant. In the latter case, we have that $\widetilde{\cN} = J H^2 \supset \cN \supset J B_c H^2$, where all of the containments are strict by assumption (since $\fc$ is an ideal of $\C[z]$ as well). Hence there exists a subspace $0 \neq M \subsetneq \C[z]/\fc$, such that $\cN = J M \oplus J B_c H^2$. Since $\cN$ is invariant under $\fA$, it then follows immediately that $M$ is in fact an $A/\fc$-submodule of $\C[z]/\fc$.
\end{proof}

By Lemma~\ref{lem:DA compression}, we may apply Proposition \ref{prop:cyclic} to the conductor ideal of a subalgebra $\fA$ of the multiplier algebra $\cM_d$. 
Therefore the cyclic $\fA$-submodules of $H^2_d$ of the form $\overline{\fA \xi}$, where $\xi$ is outer, are parametrized by cyclic $A/\fc$ submodules of $\cP_d/\fc$. Since outer functions do not vanish, the image of an outer function in $\cP_d/\fc$ is a unit. Let us denote the unit group of the Artinian ring $\cP_d/\fc$ by $\ugrp$ and the unit group of $A/\fc$ by $\augrp$. 

Now let $\xi_1, \xi_2$ be outer and assume that $\overline{\fA \xi_1} = \overline{\fA \xi_2}$. Let us denote $f_j = P_{\cC^{\perp}} \xi_j$, for $j= 1,2$. Since $\overline{\fA \xi_j} = f_j A/\fc  \oplus \cC$, implies that $f_1 A/\fc  = f_2 A/\fc $. Thus there exists a unit $u \in A/fc$, such that $f_1 = f_2 u$. Clearly, if such a unit exists, then $f_1 A/\fc = f_2 A/\fc$ and thus $\overline{\fA \xi_1} = \overline{\fA \xi_2}$. Hence we obtain the following lemma.

\begin{lem} \label{lem:first_step}
Let $\fA \subset \cM_d$ be as above, then $\fA$ has a Nevanlinna-Pick family parametrized by $\ugrp/\augrp$.
\end{lem}

We can actually give a better geometric description of the parameter space. Consider again the conductor square \eqref{eq:conductor_square} and by \cite[Equation 1.4]{ReiUll74} we have an exact sequence:
\[
\xymatrix{1 \ar[r] & A^{\times} \ar[r]^(.3)\iota & \cP_d^{\times} \times \augrp \ar[r]^\alpha & \ugrp \ar[dl]
 \\ && \Pic(A) \ar[r] & \Pic(A/\fc) \times \Pic(\C[z]).  }
\]
Here the first map is $\iota \colon \C^{\times} \cong A^{\times} \to \cP_d^{\times} \times \left(A/\fc \right)^{\times} \cong \C^{\times}\times \left(A/\fc \right)^{\times}$ is $\iota(t) = (t,t)$. The second map is $\alpha(t,f) = t f^{-1}$. Now it is clear that the image of $\alpha$ is precisely $\left( A/\fc \right)^{\times}$ and additionally since $\Pic(\cP_d) = 1$ (see for example \cite[Theorem 1.6]{GilHei80}). Finally, we note that $A/\fc$ is an Artinian ring and thus is a finite product of Artinian local rings. We conclude that $\Pic(A/\fc) =1$. This allows us to simplify the exact sequence into
\[
\xymatrix{1 \ar[r] & \left(A/\fc\right)^{\times} \ar[r] & \left(\C[z]/\fc\right)^{\times} \ar[r] & \Pic(A) \ar[r] & 1}.
\]
Thus $\Pic(A) = \left(\C[z]/\fc\right)^{\times}/\left(A/\fc\right)^{\times}$. We summarize this discussion in the following theorem.
\begin{thm} \label{thm:NP-Picard}
Let $\fA \subset \cM_d$ be a weak-$*$ closure of a subring $A \subset \cP_d$ that arises from a parametrization of a singular rational variety with isolated singular points. Then $\fA$ admits a Nevanlinna-Pick family parametrized by the Picard group of $A$.
\end{thm}

This is analogous to the situation in \cite{Abr79} as we have a collection of line bundles parametrizing the Nevanlinna-Pick families for a multiply connected domain.

\begin{example}
In the case of the algebras $\fA = \C 1 + B H^{\infty}$ considered in \cite{Rag09}, with $B =\prod_{j=1}^r \left(\frac{z - \lambda_j}{1 - z \overline{\lambda_j}}\right)^{k_j}$ a finite Blaschke product, the algebra $A = \fA \cap \C[z]$ is the algebra $A = \C 1 + f \C[z]$, where $f = \prod_{j=1}^r \left(z - \lambda_j\right)^{k_j}$ is the (monic) polynomial with the zeroes prescribed by $B$ including the order. Then $\fc = (f)$ and $A/\fc \cong \C$, whereas $\C[z]/\fc \cong \C[z]/(z-\lambda_1)^{k_1} \times \cdots \times \C[z]/(z-\lambda_r)^{k_r}$. The ring $A/\fc$ embeds into $\C[z]/\fc$ diagonally. The units of $\C[z]/\fc$ are precisely the elements with non-zero value at every $\lambda_j$. Hence as a space $\left( \C[z]/\fc\right)^{\times} \cong \underbrace{\C^{\times} \times \cdots \times \C^{\times}}_{r \text{-times}} \times \C^{k_1-1} \times \cdots \times \C^{k_r - 1}$. We conclude that the Picard group parametrizing the Nevanlinna-Pick family has dimension $ \sum_{j=1}^r k_j -1$.

In particular, for the algebra $H^{\infty}_1$ considered in \cite{DPRS09} we have that $A = \C[z^2, z^3] = \C 1 + z^2 \C[z]$ and $\Pic(A) \cong \C$. The disparity with the compact parameter space obtained in \cite{DPRS09} follows from the fact that the authors of \cite{DPRS09} consider all cyclic non-trivial $A/(z^2)$-submodules of $\C[z]/(z^2)$. This results in a one point compactification of $\Pic(A)$ yielding the complex projective space or a sphere as stated in \cite{DPRS09}. If one allows in the above setting all non-trivial cyclic modules, in other words if we replace $\left( \C[z]/\fc \right)^{\times}$ with $\C[z]/\fc \setminus \{0\}$, then one will obtain the projective space $\pp^{\sum_{j=1}^r k_j}(\C)$ that is compact.
\end{example}

\begin{example}
Consider the map $\varphi \colon \D \to \D^2$ given by $\varphi(z) = (z^2, z^5)$. In this case $A = \C[z^2,z^5]$ and $\fc = (z^4)$. Note that $\fA$ is not of the form $\C 1 + B H^{\infty}$. However, the method introduced above allows us to handle this case. We have that $A/(z^4)$ is spanned by $\{1,z^2\}$. Again the ring $\C[z]/(z^4)$ is a local Artinian ring and so is $A/(z^4)$. The invertible elements in $\C[z]/(z^4)$ are those of the form $\alpha + \beta z + \gamma z^2 + \delta z^3$, where $\alpha \neq 0$. The units of $A/(z^4)$ act on these points via
\[
(a, 0, b, 0) \cdot (\alpha,\beta, \gamma, \delta) = (a \alpha, a \beta, a \gamma + b \alpha, a \delta + b \beta).
\]
Since both $\alpha$ and $a$ are non-zero, the orbit of each point $(\alpha, \beta,\gamma,\delta)$ is the intersection of the plane spanned by $(\alpha,\beta,\gamma,\delta)$ and $(0,0,\alpha,\beta)$ and the open affine subset of $\C^4$ defined by $\alpha \neq 0$. Note also that the stabilizer of each point is trivial. The coordinate ring of this affine subset is $B = \C[\alpha,\beta,\gamma,\delta,1/\alpha]$. The quotient by the action of our group corresponds to the subring of fixed elements. It is straightforward to check that $B^{\left(A/(z^4)\right)^{\times}} = \C[\beta/\alpha,(\delta \alpha - \gamma \beta)/\alpha^2]$. The map $\C^4 \setminus \{\alpha = 0\} \to \C^2$, given by $(\alpha,\beta,\gamma,\delta) \mapsto (\beta/\alpha,(\delta \alpha - \gamma \beta)/\alpha^2)$ is surjective and thus $\C^2$ is the quotient. For each point $(x,y) \in \C^2$, we can associate an invertible element $f_{x,y} = 1 + x z + y z^3$ in $\C[z]/(z^4)$. The $A/(z^4)$-module generated by $f_{x,y}$ is spanned by $f_{x,y}$ and $z^2 f_{x,y} = z^2 + x z^3$. So we have that $H^2_{f_{x,y}} = \operatorname{Span}\{f_{x,y}, z^2 f_{x,y}\} \oplus z^4 H^2$. Let $g_{x,y}$ and $h_{x,y}$ be an orthonormal basis for the space spanned by $f_{x,y}$ and $z^2 f_{x,y}$. Then the reproducing kernel of this space is
\[
k_{x,y}(z,w) = g_{x,y}(z) \overline{g_{x,y}(w)} + h_{x,y}(z) \overline{h_{x,y}(w)} + \frac{z^4 \bar{w}^4}{1 - z \bar{w}}.
\]
If one would like a compact parameter space, one can compactify $\C^2$ by considering $\pp^2(\C)$ instead, where one identifies $\C^2$ as an open affine subset in $\pp^2(\C)$ via the map $(x,y) \mapsto (1:x:y)$. Then our rational map can be extend by $(\alpha,\beta,\gamma,\delta) \mapsto (\alpha^2:\alpha \beta: \delta \alpha - \gamma \beta)$ to the complement of the plane $\{\alpha = \beta = 0\}$. This subset corresponds to all elements $f \in \C[z]/(z^4)$, such that the $A/(z^4)$-submodule  $A/(z^4) f$ is a two dimensional vector space.

\end{example}

\begin{example}
Consider the algebra $A = \C[w,zw,z^2,z^3] \subset \cP_2 = \C[z,w]$. The map $\varphi(z,w) = \left(z^3,z^2,zw,w \right)$ is injective from $\C^2$ to $\C^4$. It is also easy to check that the only singularity of this map is at the origin. Recall that the Veronese map of degree $3$ from $\pp^2(\C)$ to $\pp^{9}(\C)$ is defined by
\[
(z:w:u) \mapsto (z^3: z^2 w: z^2 u: z w^2: z w u: z u^2: w^3: w^2 u: w u^2: u^3).
\]
Restricting to the open affine subset $u \neq 0$, we get the map
\[
(z,w) \mapsto (z^3, z^2 w, z^2, z w^2, z w, z, w^3, w^2, w, 1).
\]
Now it is easy to see that $\varphi$ is the map obtained by composition  with the projection onto the first, third, fifth and ninth coordinates. Hence the spectrum of $A$ is a projection of the above affine open subset of the degree $3$ Veronese variety.

A monomial $z^n w^m \in A$ if $m \neq 0$ and if $m =0$, then we must have $n > 1$ or $n =0$. Thus the only monomial not in $A$ is $z$. Consequently the conductor ideal is generated by $w$ and $z^2$. The quotient $\C[z,w]/\fc \cong \C[z]/(z^2)$ and $A/\fc \cong \C$. Therefore, the parameter space in this example is isomorphic to the parameter space of $H^{\infty}_1$. For each point $\alpha + \beta z$, with $|\alpha|^2 + |\beta|^2 = 1$, we get the following kernel
\[
k_{\alpha,\beta}(z,w,u,v) = (\alpha + \beta z)\overline{(\alpha + \beta u)} + \frac{w\bar{v} + z^2 \bar{u}^2 + z w \bar{u}\bar{v}}{1 -z \bar{u} - w \bar{v}}.
\]
The numerator in the fraction is obtained by observing that 
\[
M_w M_w^* + M_{z^2} M_{z^2}^* + M_{zw} M_{zw}^* = M_w M_w^* + M_z M_z^* - P_z = I - P_1 - P_z.
\]
Here $P_1$ and $P_z$ stand for the orthogonal projections onto the spaces spanned by $1$ and $z$, respectively. Hence this is the orthogonal projection onto $\overline{w H^2_2 + z^2 H^2_2}$. 

Recall from \cite[Definition 2.6]{Arv00} that for a Hilbert submodule $\cK \subset H^2_d$, a sequence $f_1,f_2,\ldots \in \cM_d$ is called an inner sequence if $\sum_{j=1}^{\infty} M_{f_j} M_{f_j}^* = P_{\cK}$ and for almost every $\zeta \in \partial \B_d$, we have $\sum_{j=1}^{\infty} |f_j(\zeta)|^2 = 1$. Consider our sequence $f_1(z,w) = w$, $f_2(z,w) = z^2$ and $f_3(z,w) = zw$ (complemented by zeroes). We have already seen that the first condition of Arveson's definition is satisfied. To see the second note that for every $\zeta = (z,w)$, such that $|z|^2 + |w|^2 =1$, we have
\[
|w|^2 + |z|^4 + |z|^2 |w|^2 = |w|^2 + |z|^2(|z|^2 + |w|^2) = |w|^2 + |z|^2 = 1.
\]
Hence this is an inner sequence for the submodule $\overline{w H^2_2 + z^2 H^2_2}$.
\end{example}

\section{Nevanlinna-Pick Families --- The Matrix Case} \label{sec:matrix_case}

First, we need an analog of \cite[Theorem 3.1]{DavHam11} in the setting of the Drury-Arveson space. In \cite[Lemma 7.4]{DavHam11} the first author and Hamilton prove this result for $d = 1$. The proof of the following proposition, just like the proof of \cite[Theorem 3.1]{DavHam11}, is based on \cite{DavPit99}.

\begin{prop} \label{prop:matrix_a11}
Let $k \in \N$, then $M_k(\cM_d)$ acting on $M_k(H^2_d)$ by transposed right multiplication, i.e, the map $M_k(\cM_d) \to M_{k^2}(\cM_d)$ is given by $F \mapsto F \otimes I_k$,  has $\A_1(1)$. More precisely, for every weak-$*$ continuous functional $\varphi$ on $M_k(\cM_d)$ with $\|\varphi\| < 1$, there exist $X,Y$ in $\left(H^2_d\right)^{\oplus ks}$ such that $s \leq k$, $\|X\|, \|Y\| < 1$ and $\varphi(F) = \langle (F \otimes I_s) X, Y \rangle$. Moreover, $X$ can be chosen to be $M_{k}(\cM_d)$-cyclic.
\end{prop}

\begin{proof}
Let $\cF^2_d = \oplus_{n=0}^{\infty} \left( \C^d \right)^{\otimes n}$ be the full Fock space. The symmetrization map is a projection $P \colon \cF^2_d \to H^2_d$. The cutdown by $P$ induces a complete contractive and weak-$*$ continuous map $q \colon \cL_d \to \cM_d$, where $\cL_d$ stands for the weak-$*$ closed algebra generated by the left creation operators on $\cF^2_d$. Let $q_k \colon M_k(\cL_d) \to M_k(\cM_d)$ denote $q \otimes I_{M_k}$. Any weak-$*$ continuous functional $\varphi$ on $M_k(\cM_d)$ can be pulled-back to $M_k(\cL_d)$ without increasing the norm ($\varphi \circ q_k$). By \cite{DavPit99} the algebra $M_k(\cL_d)$ has $\A_1(1)$. Hence we can find two vectors $\xi, \eta \in \left(\cF^2_d\right)^{\oplus k}$, such that $\varphi \circ q_k(T) = \langle T \xi, \eta \rangle$, for every $T \in M_k(\cL_d)$. Let $\xi = \left( \xi_1,\ldots,\xi_k \right)^T$ and let $\cK = \sum_{j=1}^k \cL_d \xi_j$ be the $\cL_d$-invariant subspace generated by the coordinates of $\xi$. By \cite[Theorem 2.1]{DavPit99}, there exists a row isometry $R \colon \left(\cF_d^2\right)^{\oplus s} \to \cF^2_d$ with coordinates in $\cR_d = \cL_d'$, such that $\cK = R \left(\cF_d^2\right)^{\oplus s}$ and $s \leq k$. Set $\xi_j = R u_j$, for $1 \leq j \leq k$ and $u_1,\ldots,u_k \in \left(\cF_d^2\right)^{\oplus s}$. Additionally, let $\eta = \left(\eta_1,\ldots,\eta_k \right)^T$ and set $v_j = R^* \eta_j$, for $1 \leq j \leq k$. Note that $u_1,\ldots, u_k$ generate $\left(\cF_d^2\right)^{\oplus s}$ as a $\cL_d$-module. Hence if we let $U$ be the column vector obtained by stacking the $u_j$ and similarly $V$ is the column vector of the $v_j$, then
\[
\varphi \circ q_k(F) = \langle F \xi, \eta \rangle = \langle F (I_k \otimes R) U, \eta \rangle = \langle F \otimes I_s) U, V \rangle.
\]
Now set $X = P U$ and $Y = P V$. Since $u_1,\ldots,u_k$ generate the $\cL_d$-module $\left(\cF_d^2\right)^{\oplus s}$, then $U$ is $M_k(\cM_d) \otimes I_s$-cyclic and hence so is $X$.
\end{proof}

\begin{cor} \label{cor:vec_to_mat}
In the proof above one may consider $U$ as a matrix with columns $u_j$ and similarly $V$ is a matrix with columns $v_j$. Let $X = P U$ and $Y = P V$. Then for every $F \in M_k(\cM_d)$ we have $\varphi(F) = \tr\left( X F^T Y^* \right)$.
\end{cor}

\begin{lem} \label{lem:inj}
The rank of the matrix $X(z)$, for $z \in \B_d$ is constant and is equal to $s$, where $s$ is the number obtained in Proposition \ref{prop:matrix_a11}.
\end{lem}
\begin{proof}
We need to show that for $z \in \B_d$, the matrix $X(z)$ is surjective. Then there exists $w \in \C^s$ non-zero, such that $\langle x_j(z), w \rangle = 0$ for all $1 \leq j \leq k$, where $x_j$ are the columns of $X$. That, however, contradicts the fact that the $x_j$ generate $\left( H^2_d \right)^{\oplus s}$ as an $M_s(\cM_d)$-module.
\end{proof}

\begin{lem} \label{lem:cyclic_mat_mod}
Let $X \in M_{k,s}(H^2_d)$ be a $M_k(\cM_d) \otimes I_s$-cyclic vector and $H^2_{d,X} = \overline{\left(M_k(\fA) \otimes I_s \right) X}$. If $\widetilde{X} = P_{\cC^{\perp}} X$, then $$H^2_{d,X} = \left(M_k(A/\fc) \otimes I_s \right) \widetilde{X} \oplus M_{k,s}(\fC).$$ Furthermore, the map $\widetilde{X} \colon \left(\cP_d/\fc\right)^{\oplus k} \to \left(\cP_d/\fc\right)^{\oplus s}$ is surjective.
\end{lem}
\begin{proof}
To see the first part of the lemma one simply notes that 
\[
\overline{\left(M_k(\overline{\fc}^{w*}) \otimes I_s\right) X} = \overline{\left(M_k(\overline{\fc}^{w*}) \otimes I_s\right) \left(M_k(\cM_d) \otimes I_s\right) X} = M_{k,s}(\fC).
\]
The second follows from Lemma \ref{lem:inj} and the Nakayama lemma \cite[Corollary 4.8]{Eis95}.
\end{proof}

Now given $X_1, X_2 \in M_{k,s}(H^2_d)$ that are $M_k(\cM_d) \otimes I_s$-cyclic, we ask when is $H^2_{d,X_1} = H^2_{d,X_2}$? This is if and only if there exists $F,G \in M_k(A/\fc)$, such that $\widetilde{X_1} = \widetilde{X_2} F^T$ and $\widetilde{X_2} = \widetilde{X_1} G^T$. In particular, this is true if we can find $F \in \GL_k(A/\fc)$, such that $\widetilde{X_1} = \widetilde{X_2} F^T$. However, if $s < k$, then it need not be the case. 

Set $\cQ_s = \{ Z \in \Hom(\left(\cP_d/\fc\right)^{\oplus k},\left(\cP_d/\fc\right)^{\oplus s} \mid \rank(Z) = s\}$. Then from \cite[Theorem 7.2]{DavHam11} we have that

\begin{thm} \label{thm:matrix_family}
Let $\fA \subset \cM_d$ be a weak-$*$ closure of a subring $A \subset \cP_d$ that arises from a parametrization of a singular rational variety with isolated singular points. Then for every $k \in \N$, the algebra $M_k(\fA)$ admits a Nevanlinna-Pick family parametrized by the space $\sqcup_{s < k} \cQ_s/\left(\GL_k(A/\fc) \otimes I_s\right)$.
\end{thm}

\begin{rem}
Let us fix $k \in \N$ and $s \leq k$. Let $Z \in \cQ_s$. Consider the columns of $Z^T$ as elements of $\left(\cP_d/\fc\right)^{\oplus k}$ and let $M_Z$ be the $A/\fc$-module generated by these elements. By the assumption $Z$ is surjective and hence $\cP/\fc \otimes_{A/\fc} M_Z$ is a free $\cP_d/\fc$-module of rank $s$. Hence this data corresponds to a rank $s$ vector bundle on the spectrum of $A$ as in \cite{BDG01}. However, we might get isomorphic vector bundles by considering different submodules of $\left(\cP_d/\fc\right)^{\oplus s}$. Therefore, in a sense, the Nevanlinna-Pick family is parametrized by all vector bundles on the spectrum of $A$ of rank less than or equal to $k$.
\end{rem}

\section{Property $\A_1(1)$ for matrices} \label{sec:a11_mat}

In this section we study weak-$*$ closed operator algebras $\fA \subset B(\cH)$, such that $M_k(\fA)$ has the proeprty $\A_1(1)$ acting on $\cH^{\oplus k}$. First we provide a generalization of \cite[Corollary 3.5]{BFP83} with bounds on the condition number of the similarity.

\begin{thm} \label{thm:a_11_for_mat}
Let $\fA \subset B(\cH)$ be a weak-$*$ closed unital operator algebra. Assume that $M_k(\fA)$ has $\A_1(1)$ acting on $\cH^{\oplus k}$. Let  $\varphi \colon \fA \to M_k$ be a weak-$*$ continuous completely contractive unital homorphism. Then for $\epsilon > 0$, there exists a $k$-dimensional subspace $\cK_{\epsilon} \subset \cH$ that is semi-invariant for $\fA$, and an invertible linear map $S_{\epsilon} \colon \C^k \to \cK_{\epsilon}$, such that $\varphi(a) = S_{\epsilon}^{-1} P_{\cK_{\epsilon}} a|_{\cK_{\epsilon}} S_{\epsilon}$ for every $a \in \fA$, and $\lim_{\epsilon \to 0+} \|S_{\epsilon}\| \|S_{\epsilon}^{-1}\| = 1$.
\end{thm}

\begin{proof}
Let $e_1,\ldots, e_k$ be an orthonormal basis for $\C^k$ and let $\varphi_{ij}(a) = \langle \varphi(a) e_j, e_i \rangle$ be the matrix coefficients of $\varphi$. Since $\varphi$ is completely contractive and unital it extends to a unital completely positive (ucp) map on $\fA + \fA^*$. We can now construct a state on $M_k(\fA)$ from $\varphi$. Let $T = \sum_{i,j = 1}^k E_{ij} \otimes a_{ij} \in M_k(\fA)$, then
\begin{align*}
s(T) &= \frac{1}{k} \sum_{i,j = 1}^k \varphi_{ij}(a_{ij}) \\&
= \frac{1}{k} \sum_{i,j =1}^k \tr \left(\varphi(a_{ij}) E_{ij}^* \right)  \\&
= \langle \left(\operatorname{id}_{M_k} \otimes \varphi\right)(T) I_k, I_k \rangle.
\end{align*}
In the latter equality, we view $M_k$ as a Hilbert space with the normalized Hilbert-Schmidt product.

Fix $\epsilon > 0$. Since $s$ is a state and $M_k(\fA)$ has $\A_1(1)$, there exist $\xi, \eta \in \cH^{\oplus k}$, such that $s(T) = \langle T \xi, \eta \rangle$ for every $T \in M_k(\fA)$ and $\|\xi\|, \|\eta\| < \sqrt{1 + \epsilon}$. Write $\xi = \left( \xi_1,\ldots,\xi_k \right)^T$ and $\eta = \left( \eta_1,\ldots, \eta_k \right)^T$. Note that for every $1 \leq i,j \leq k$, 
\[
\frac{1}{k} \varphi_{ij}(a) = s(E_{ij} \otimes a) = \langle \left(E_{ij} \otimes a\right) \xi, \eta \rangle = \langle a \xi_j, \eta_i \rangle.
\]
Following \cite{BFP83} we define the following subspaces of $\cH$
\begin{alignat*}{2}
& \cN_{\epsilon} = \overline{\sum\!\strut_{j=1}^k \fA \xi_j}, &\qquad& \cN_{*,\epsilon} = \overline{\sum\!\strut_{j=1}^k \fA^* \eta_j}, \\[.5ex]
& \cG_{\epsilon} = \cN_{\epsilon} \cap \cN_{*,\epsilon}, &\qquad& \cK_{\epsilon} = \cN_{\epsilon} \ominus \cG_{\epsilon}.
\end{alignat*}
Clearly $\cK_{\epsilon}$ is a semi-invariant subspace of $\fA$. Now write $\xi_j = \xi_{j1} + \xi_{j2}$, with $\xi_{j1} \in \cK_{\epsilon}$ and $\xi_{j2} \in \cG_{\epsilon}$. Fix $j$, then for every $a \in \fA$ and every $1 \leq i \leq k$ we have
\[
\frac{1}{k}\varphi_{ij}(a) = \langle a \xi_j, \eta_i \rangle = \langle \xi_j, a^* \eta_i \rangle = \langle a \xi_{j1}, \eta_i \rangle.
\]
Hence we can assume that $\xi_1,\ldots,\xi_k \in \cK_{\epsilon}$. Since $\xi_{1j}$ is a projection of $\xi_j$, we will not increase the norm by replacing $\xi_j$ with $\xi_{1j}$. From the fact that $\varphi$ is a homomorphism we have that for every $a, b \in \fA$ and every $1 \leq i,j \leq k$
\begin{align*}
\langle a \xi_j, b^* \eta_i \rangle &= \langle ba \xi_j, \eta_i \rangle = \frac{1}{k} \varphi_{ij}(ba) 
= \frac{1}{k} \sum_{r=1}^k \varphi_{ir}(b) \varphi_{rj}(a) \\&
= k \sum_{r=1}^k \langle \xi_r, b^* \eta_i \rangle \langle a \xi_j,  \eta_r \rangle = \langle k \sum_{r=1}^k \langle a \xi_j, \eta_r \rangle \xi_r, b^* \eta_i \rangle
\end{align*}
This immediately implies that 
\[
P_{\cK_{\epsilon}} a \xi_j = k \sum_{r=1}^k \langle a \xi_j, \eta_r \rangle \xi_r .
\]
Additionally, since the elements of the form $\sum_{r=1}^k P_{\cK_{\epsilon}} a_r \xi_r$ are dense in $\cK_{\epsilon}$, we obtain that $\cK_{\epsilon}$ is spanned by $\xi_1,\ldots,\xi_k$. Since $\varphi$ is unital we have that $\langle \xi_j, \eta_i \rangle = \frac{1}{k} \delta_{ij}$. Conclude that $\xi_1,\ldots,\xi_k$ is a basis for $\cK_{\epsilon}$.

The last consideration can be slightly refined, namely
\begin{equation} \label{eq:almost_cs}
1 = s(I) = \langle \xi, \eta \rangle = \sum_{r=1}^k \langle \xi_r, \eta_r \rangle \leq \|\xi\| \|\eta\| \leq 1 + \epsilon.
\end{equation}

Now consider the map $S_{\epsilon} \colon \C^k \to \cK$ given by 
\[
 S_\epsilon e_i =  \xi_i \quad\text{for }1 \le i \le n .
\]
It is immediate that $S_{\epsilon}$ is bijective. 
Furthermore, $S_{\epsilon}$ intertwines $\varphi(a)$ and $P_{\cK_{\epsilon}} a|_{\cK_{\epsilon}}$. To see this we simply observe that
\begin{align*}
S_{\epsilon} \varphi(a) e_i &= \sum_{r=1}^k \varphi_{ri}(a) S_{\epsilon} e_r \\&
= k \sum_{r=1}^k \langle a \xi_i, \eta_r \rangle \xi_r 
= P_{\cK_{\epsilon}} a \xi_i \\&
= P_{\cK_{\epsilon}} a S_{\epsilon} e_i .
\end{align*}
Hence every homomorphism is similar to the compression of $\fA$ to a semi-invariant subspace. It remains to estimate the condition number. Note that as $\epsilon$ tends to $0$ in \eqref{eq:almost_cs} we have that $\eta$ tends to $\xi$. In particular, the norm of each $\xi_j$ tends to $\frac{1}{\sqrt{k}}$ and the inner products $\langle \xi_i, \xi_j\rangle$ for $i \neq j$ tend to $0$. Therefore $\sqrt k S_{\epsilon}$ is close to a unitary, and thus the condition number is close to $1$.
\end{proof}

\begin{cor}
Let $\fA \subset B(\cH)$ be a unital weak-$*$ closed operator algebra. Assume that $M_k(\fA)$ has $\A(1)$ acting on $\cH^{\oplus k}$. Let $\varphi \colon \fA \to M_k$ be a weak-$*$ continuous unital homomorphism.  Then for every $\epsilon > 0$, there exists a semi-invariant subspace $\cK_{\epsilon} \subset \cH$ and an invertible linear map $T_{\epsilon} \colon \C^k \to \cK_{\epsilon}$, such that for every $a \in \fA$, $\varphi(a) = T_{\epsilon}^{-1} P_{\cK_{\epsilon}} a|_{\cK_{\epsilon}}$ and $\lim_{\epsilon \to 0+} \|T_{\epsilon}\| \|T_{\epsilon}^{-1}\| = \|\varphi\|_{cb}$.
\end{cor}
\begin{proof}
By Smith's lemma every bounded homomorphism $\varphi \colon \fA \to M_k$ is completely bounded and by a theorem of Paulsen it is similar to a completely contractive homomorphism. Furthermore, there is a similarity $S$, such that $\psi = S^{-1} \varphi S$ is completely contractive and $\|\varphi\|_{cb} = \|S\| \|S^{-1}\|$. Given $\epsilon > 0$, we apply Theorem \ref{thm:a_11_for_mat} to $\psi$ to find a semi-invariant subspace $\cK_{\epsilon} \subset \cH$ and an invertible linear map $S_{\epsilon} \colon \C^k \to \cK_{\epsilon}$, such that for every $a \in \fA$, $\psi(a) = S_{\epsilon}^{-1} P_{\cK_{\epsilon}} a|_{\cK_{\epsilon}} S_{\epsilon}$. Set $T_{\epsilon} = S_{\epsilon} S^{-1}$. It is clear now that $T_{\epsilon}$ is the similarity that we are after. To prove the last statement we note that 
\[
\|\varphi\|_{cb} \leq \|T_{\epsilon}\|\|T_{\epsilon}^{-1}\| \leq \|S\|\|S^{-1}\|\|S_{\epsilon}\|\|S_{\epsilon}^{-1}\| = \|\varphi\|_{cb} \|S_{\epsilon}\|\|S_{\epsilon}^{-1}\|.
\]
Now it remains to apply the fact that $\lim_{\epsilon \to 0+} \|S_{\epsilon}\|\|S_{\epsilon}^{-1}\| = 1$.
\end{proof}

\newcommand{\qand}{\quad\text{and}\quad}

The following lemma and corollary demonstrate that the methods employed in the previous section to obtain the property $\A_1(1)$ for $M_k(\cM_d)$ are necessary.

\begin{lem} \label{lem:must_be_close}
Suppose that $f\in H^2_2$ satisfies 
\[ \|f\|=1 \qand  \|z_1 f\|^2 > 1-\epsilon \qand \| z_2f\|^2 > 1-\epsilon \]
Then $|\langle f, 1 \rangle|^2 \ge 1-4\epsilon$ and $\operatorname{dist}(f,\C1) \le 2\sqrt\epsilon$.

\end{lem}

\begin{proof}
Write $f = \sum a_{mn} z_1^m z_2^n$. 
Note that the monomials are orthogonal and 
$\| z_1^m z_2^n \|^2 = \dfrac{m!\, n!}{(n+m)!}$. So we have
\[ 1 = \|f\|^2 = \sum |a_{mn}|^2 \frac{m!\, n!}{(n+m)!} . \]
Thus
\begin{align*}
 \| z_1f\|^2 + \| z_2f \|^2 &= \sum |a_{mn}|^2  \frac{m!\, n!}{(n+m)!} \Big( \frac{m+1}{m+n+1} + \frac{n+1}{m+n+1} \Big) \\&
 = 1 + \sum |a_{mn}|^2  \frac{m!\, n!}{(n+m)!}  \frac 1{m+n+1} \\ & \leq \frac{3}{2} + \frac{1}{2} |a_{00}|^2 .
\end{align*}
Hence we have that
\[
1 - 4 \epsilon \leq |a_{00}|^2 = |\langle f, 1 \rangle|^2 \leq 1.
\]
It follows that
\[
\operatorname{dist}(f,\C1)^2 = \|f\|^2 -  |\langle f, 1 \rangle|^2 \le 4\epsilon. \qedhere
\]
\end{proof}

\begin{cor}
The algebra $M_6(\cM_d)$ does not have $\A_1(1)$. 
\end{cor}

\begin{proof}
By Theorem \ref{thm:a_11_for_mat} we need to produce a weak-$*$ continuous completely contractive homomorphism $\varphi \colon \cM_d \to M_6$ which cannot be approximated by compressions to semi-invariant subspaces. We will prove this for $\cM_2$, and for every other $d > 2$, it follows by embedding $\cM_2$ into $\cM_d$.

Let us consider first the compression of $\cM_2$ to $\operatorname{Span}\{1,z_1,z_2\}$. This is a semi-invariant subspace and hence this is a representation that we will denote by $\pi$. It is clear that $\pi$ unital, completely contractive, and weak-$*$ continuous. Let $\varphi = \pi \oplus \pi$. In order to approximate $\varphi$ by compressions to semi-invariant subspaces, we need to be able to find for every $\epsilon > 0$, two unit vectors $f,g \in \cH_2^2$, such that $| \langle f , g \rangle| < \epsilon$ and the compression of $\cM_2$ onto $\operatorname{Span}\{f, z_1 f, z_2 f\}$ and onto $\operatorname{Span}\{g, z_1 g, z_2 g\}$ is $\epsilon$-similar to $\pi$. In particular, this implies that both $f$ and $g$ satisfy the assumptions of Lemma \ref{lem:must_be_close}. Multiplying $f$ and $g$ by unimodular scalars will not change the properties of these vectors. Hence we may assume that $f(0), g(0) > 0$.  However, this implies that 
\[
\sqrt{1 - 4 \epsilon} \leq \langle f , 1 \rangle, \, \langle g, 1 \rangle \leq 1.
\]
Therefore
\begin{align*}
|\langle f, g \rangle| &\ge \langle f , 1 \rangle  \langle g, 1 \rangle - \| P_{\C1}^\perp f\|\,\|P_{\C1}^\perp g\| \\&
\ge (1-4\epsilon) - (2\sqrt\epsilon)^2 = 1 - 8\epsilon .
\end{align*}
This contradicts the fact that $\langle f, g \rangle < \epsilon$ when $\epsilon < 0.1$.
\end{proof}

We suspect that $M_2(\cM_d)$ does not have $\A_1(1)$. We state this as an explicit problem:

\begin{quest}
Does $M_2(\cM_d)$  have $\A_1(1)$ for $d\ge2$?
\end{quest}


\end{document}

%% file: operators.tex
\usepackage{mathrsfs}

\newcommand{\Hom}{\operatorname{Hom}} 

\newcommand{\rank}{\operatorname{rank}} 

\newcommand{\GL}{\operatorname{{\mathbf GL}}}

\newcommand{\tr}{\operatorname{tr}}


%% file: fonts.tex
\newcommand{\cC}{{\mathcal C}}

\newcommand{\cF}{{\mathcal F}}
\newcommand{\cG}{{\mathcal G}}
\newcommand{\cH}{{\mathcal H}}
\newcommand{\cI}{{\mathcal I}}
\newcommand{\cJ}{{\mathcal J}}
\newcommand{\cK}{{\mathcal K}}
\newcommand{\cL}{{\mathcal L}}
\newcommand{\cM}{{\mathcal M}}
\newcommand{\cN}{{\mathcal N}}

\newcommand{\cP}{{\mathcal P}}
\newcommand{\cQ}{{\mathcal Q}}
\newcommand{\cR}{{\mathcal R}}

\newcommand{\fA}{{\mathfrak A}}
\newcommand{\fB}{{\mathfrak B}}
\newcommand{\fC}{{\mathfrak C}}

\newcommand{\fc}{{\mathfrak c}}

\newcommand{\fm}{{\mathfrak m}}

\newcommand{\fp}{{\mathfrak p}}
\newcommand{\fq}{{\mathfrak q}}

\newcommand{\A}{{\mathbb A}}

\newcommand{\C}{{\mathbb C}}
\newcommand{\R}{{\mathbb R}}
\newcommand{\pp}{\mathbb{P}}

\newcommand{\N}{{\mathbb N}}
\newcommand{\Z}{{\mathbb Z}}

\newcommand{\T}{{\mathbb T}}
\newcommand{\D}{{\mathbb D}}
\newcommand{\B}{{\mathbb B}}

